\documentclass[12pt]{amsart}

\newtheorem{theorem}{Theorem}[section]

\newtheorem{corollary}[theorem]{Corollary}

\newtheorem{lemma}[theorem]{Lemma}

\theoremstyle{definition}

\newtheorem{definition}[theorem]{Definition}

\newtheorem{remark}[theorem]{Remark}

\newtheorem{example}[theorem]{Example}

\theoremstyle{parrafo}

\begin{document}

\title[]{Local comparability of measures, averaging and  maximal averaging operators}

\author{J. M. Aldaz}
\address{Instituto de Ciencias Matem\'aticas (CSIC-UAM-UC3M-UCM) and Departamento de 
Matem\'aticas,
Universidad  Aut\'onoma de Madrid, Cantoblanco 28049, Madrid, Spain.}
\email{jesus.munarriz@uam.es}
\email{jesus.munarriz@icmat.es}


\thanks{The author was partially supported by Grant MTM2015-65792-P of the
MINECO of Spain, and also by by ICMAT Severo Ochoa project SEV-2015-0554 (MINECO)}







\begin{abstract} We explore the consequences for the boundedness properties of averaging and
maximal averaging operators, of the following local comparabiliity 
 condition for measures: Intersecting balls of the same radius
have comparable sizes.  Since in geometrically doubling spaces this property 
yields the same results as doubling, we study under which circumstances it is equivalent to
the latter condition, and when it is more general.
We also study the concrete case of the standard gaussian measure, where this
property fails, but nevertheles averaging operators are uniformly bounded, with respect to the radius,  in $L^1$.
However, such bounds grow exponentially fast with the dimension, very much unlike the case of Lebesgue measure.
\end{abstract}


\maketitle


\markboth{J. M. Aldaz}{Local comparability of measures}

\section {Introduction} 

In the development of analysis in metric spaces, the doubling condition on a measure
has played a considerable role, cf. for instance \cite{He}, \cite{HKST} and the references contained
therein. However, the existence of a 
doubling measure imposes severe restrictions on the growth of the  spaces
under consideration: They must be geometrically doubling (of homogeneous type in the
terminology of \cite{CoWe1}, cf. Definition \ref{geomdoub} below; we always assume that measures are not
identically 0). This excludes many spaces
of interest such as,
for instance, hyperbolic spaces, as well as several other manifolds with  negative curvature.

Hence, there have been efforts to remove or at least weaken the doubling 
condition whenever 
possible. A considerable amount of work has been made in this area, regarding singular
integrals and Calder\'on-Zygmund operators
cf., for instance,
\cite{To}  and the references contained
therein. Here we are interested in the boundedness properties of the Hardy-Littlewood
maximal function  defined by a Borel measure $\mu$. Good boundedness results
 appear to be related to the following property,
studied in this paper: Intersecting balls
with the same radius have a comparable size (cf. Definition \ref{loccomp} below for the precise statement). 
Since this hypothesis
 does not apply (at least, not directly) to balls that fail to intersect, we shall say that
 $\mu$ satisfies a {\em local} comparability condition, even though it applies to balls both large and
 small, and not just to small balls. 
 
 The idea of using local comparabilitiy  instead
 of doubling is suggested by \cite{NaTa}, where this ``uniformity condition" (as is called there, cf. 
 p. 737) is combined with the notion of microdoubling to define ``strong microdoubling".
 But local comparability by itself is worthy of study, independently of any microdoubling conditions.
 In geometrically doubling spaces,
 this more precise hypothesis yields the same results as doubling, and is sometimes equivalent
 to it. But in general, it is satisfied by a wider class of measures.

 A second source of interest comes from attempts to understand which factors influence
 the size of bounds for averaging and maximal averaging operators. This leads us
 to consider measures for which local comparability is missing. But even for a doubling measure,
 if one is interested in quantitative aspects of the bounds,
 one may want to keep track of the local 
 comparability constant,  which often will be much smaller than the doubling
 constant (for instance, 1 and $2^d$ for $d$-dimensional Lebesgue measure).
 
 This paper is organized  as follows. In Section 2 we review some standard
 definitions and facts,  introducing also the terminology of blossoms. 
 The $r$-blossom of a
 set is the just its $r$-neighborhood, but we find the new terminology more convenient
 when talking about properties of measures. In particular, to carry out the usual Vitali covering
 argument, instead of doubling it is enough to assume that the measure blossoms boundedly
 (cf. Definition \ref{blossom} below). 
 
It is natural to ask which properties of the measure determine, or have
an influence, in the behaviour of the bounds satisfied by averaging and 
maximal averaging operators. In this regard,
  Section 3 contains the definition of local comparability, and Section 4 considers  averaging operators when local comparability is
 missing. In general metric spaces, without local comparability averaging operators
 may fail to be bounded for all $p < \infty$. Nevertheless, in  the special case of $\mathbb{R}^d$ with
 the standard gaussian measure,  we show that 
 $\sup_{r>0}\|A_r\|_{L^1\to L^1} \le (2 + \varepsilon)^d$ for every $\varepsilon > 0$ and $d$ large enough. 
However, lack of local comparability
makes itself felt in the fact that these bounds  grow exponentially fast 
with the dimension,  for all $1 \le p < \infty$.
Using an argument of A. Criado and P. Sj\"ogren, we show that 
for  every $p$ in $[1,\infty)$ and every $d$ sufficiently large, the weak type $(p,p)$
constants satisfy $\|A_{\frac{\sqrt{3d - 3}}{2}} \|_{L^p\to L^{p, \infty}} >1.019^{d/p}$. 

In the case of Lebesgue measure in $\mathbb{R}^d$,  E. M. Stein showed that for the centered
maximal function $M$  associated to euclidean balls,
the best strong type $(p,p)$ bounds  are independent of $d$, and hence of the doubling constant
$2^d$  (\cite{St1}, \cite{St2},
\cite{StSt}, see also \cite{St3}). In fact, for $p \ge 2$, P. Auscher and M. J. Carro gave the explicit  bound
$\|M\|_{L^p\to L^p} 
\le (2 + \sqrt{2})^{2/p}$ (\cite{AuCa}).
Comparing the situation with the gaussian measure, we see that
taking the supremum over radii can have a much smaller impact on the size of constants than considering 
measures without local comparability.

Stein's result was
generalized to the maximal function defined using an arbitrary norm by J. Bourgain (\cite{Bou1},
\cite{Bou2}) and A. Carbery (\cite{Ca}) when $p>3/2$.
For $\ell_q$ balls, $1\le q <\infty$, D. M\"{u}ller \cite{Mu}
showed that uniform bounds again hold for every $p
> 1$ (given $1\le q <\infty$, the $\ell_q$ balls
are defined using the norm $\|x\|_q :=\left( |x_1|^q+ |x_2|^q+\dots + |x_d|^q\right)^{1/q}$). Finally,
in  \cite{Bou3}, Bourgain proved that for cubes (balls with respect to the $\ell_\infty$ norm)  the uniform
bounds hold for every $p > 1$. Since  for cubes it is known that the weak type (1,1) constants diverge to
infinity (cf. \cite{A}, and  for the highest lower bounds currently known, cf.  \cite{IaSt}) 
this is the only case where a fairly complete picture is available. Now if the local comparability constant
is a key factor here, then 
one would expect that maximal functions defined using Lebesgue measure and different balls,
would all behave in a similar way, both regarding weak and strong type constants. 
But I have made no progress in this direction.  
 
Since in geometrically doubling 
metric spaces local comparability yields the same boundedness results as doubling, 
it is natural to enquire to what extent the first property is more general than the second. 
This is done in Section 4, where among other results, it is proven that for  geometrically doubling spaces that are
 quasiconvex, or  have the approximate midpoint property, or where  all the  balls are connected, 
 local comparability is equivalent to doubling. In particular, this is the case for euclidean spaces. 
 However,
we shall also see that there is an arc-connected, geometrically doubling 
metric space, with a non-doubling measure satisfying a local comparability
condition. Generally speaking, in spaces with poor conectivity properties, or with
large gaps (as is often the case, for instance, with fractals), or where the
``intrinsic" and ``ambient" metrics are not comparable, the two classes of measures
can be quite different. 

While the existence of a doubling measure imposes restrictions on the growth of the space, this is
not the case with local comparability, which is just a uniformity condition. It may well be that local comparabity
 can yield positive results beyond geometrically doubling spaces,
but I have made no progress in this direction.

\section {Notation and background material} 

We will use $B(x,r) := \{y\in X: d(x,y) < r\}$ to denote open balls, 
$\overline{B(x,r)}$ to denote their topological closures, and 
$B^{cl}(x,r) := \{y\in X: d(x,y) \le r\}$ to refer to closed balls (consider
$\overline{B(0,1)}$ in $\mathbb{Z}$ to see the difference).

\begin{definition} We say that $(X, d, \mu)$ is a {\em metric measure space} if
$\mu$ is a Borel measure on the metric space $(X, d)$, such that for all
balls $B(x,r)$, $\mu (B(x,r)) < \infty$, and furthermore, $\mu$ is {\em $\tau$-smooth}. 
A Borel measure is $\tau$-smooth if for every
collection  $\{U_\alpha : \alpha \in \Lambda\}$
 of  open sets, $\mu (\cup_\alpha U_\alpha) = \sup \mu(\cup_{i=1}^nU_{\alpha_i})$,
 where the supremum is taken over all finite subcollections of $\{U_\alpha : \alpha \in \Lambda\}$.
\end{definition} 

In separable metric spaces, arbitrary unions of open sets can be reduced to
countable unions, so $\tau$-smoothness is an
immediate consequence of  countable additivity; trivially also, all Radon measures are $\tau$-smooth.
The hypothesis of $\tau$-smoothness is rather weak, since it is consistent with standard set theory
(Zermelo-Fraenkel with Choice) that  in every metric space, every Borel measure  which
assigns finite measure to balls  is $\tau$-smooth 
(cf. \cite[Theorem (a), pg. 59]{Fre}).  Thus, in standard mathematical practice we will never
encounter an example where $X$ is metric and $\mu$ fails to be $\tau$-smooth.

\begin{definition}\label{maxfun} Let $(X, d, \mu)$ be a metric measure space and let $g$ be  a locally integrable function 
on $X$. For each fixed $r > 0$ and each $x\in X$ such that $0 < \mu (B(x, r))$, the
averaging operator $A_{r, \mu}$ is defined as
\begin{equation}\label{avop}
A_{r , \mu} g(x) := \frac{1}{\mu
(B(x, r))} \int _{B(x, r)}  g d\mu.
\end{equation}
In addition, the centered Hardy-Littlewood maximal operator $M_{\mu}$ is given by
\begin{equation}\label{HLMFc}
M_{\mu} g(x) := \sup _{\{r > 0:  \mu B(x, r)  > 0\}} A_{r , \mu} |g|(x),
\end{equation}
while the uncentered Hardy-Littlewood maximal operator $M^u_{\mu}$ is defined via
\begin{equation}\label{HLMFu}
M^u_{\mu} g(x) := \sup _{\{r > 0, y\in X :  \  d(x,y) < r \mbox{ and } \mu B(y, r)  > 0\}} A_{r , \mu} |g|(y).
\end{equation}
\end{definition}

According to our convention, averaging operators are defined almost everywhere (since by
$\tau$-smooth\-ness the complement of the support has measure zero) while maximal
operators are defined everywhere, for given any $x\in X$ there exists an $r >0$ such
that $\mu B(x,r) > 0$. Also, maximal operators can be defined using closed balls instead of open balls,
and this does not change their values, since open balls can be approximated from within
by closed balls, and closed balls can be approximated from without by open balls. When
the measure is understood, we will omit the subscript $\mu$  from 
$A_{r,\mu}$,  $M_\mu$, and $M^u_\mu$. 

For a given $p$ with $1 \le p < \infty$,
an operator $T$ satisfies a
weak type $(p,p)$ inequality if there exists a constant $c > 0$ such that
\begin{equation}\label{weaktypep}
\mu (\{T g \ge \alpha\}) \le \left(\frac{c \|g\|_{L^p(\mu)}}{\alpha}\right)^p,
\end{equation}
where $c=c(p,  \mu)$ depends neither on $g\in L^p (\mu)$
nor on $\alpha > 0$. The lowest constant $c$ that satisfies the preceding
inequality is denoted by $\|T\|_{L^p\to L^{p, \infty}}$.
Likewise, if there exists a constant $c > 0$ such that
\begin{equation}\label{strongtypep}
\|T g \|_{L^p(\mu)}  \le c \|g\|_{L^p(\mu)},
\end{equation}
we say that $T$ satisfies a
strong type $(p,p)$ inequality.  The lowest such constant (the operator norm) 
 is denoted by $\|T\|_{L^p\to L^{p}}$.

\begin{definition} A Borel measure $\mu$ on $(X,d)$ is {\em doubling}  if there exists a 
$C> 0 $ such that for all $r>0 $ and all $x\in X$, $\mu (B(x, 2 r)) \le C\mu(B(x,r)) < \infty$. 
\end{definition}

The following definition comes essentially from \cite[p.  739]{NaTa}, but the terminology is new.

\begin{definition} \label{blossom} Given a set $S$ we define its $s$-{\em blossom} as the enlarged set
\begin{equation} \label{altblossom}
Bl(S, s):= \cup_{x\in S}B(x,s),
 \end{equation}
 and its  {\em uncentered $s$-blossom} as the set
\begin{equation} \label{altublossom}
Blu(S, s):= \cup_{x\in S}\cup\{B(y, s): x\in B(y, s)\}.
 \end{equation}
When $S= B(x,r)$, we simplify the notation and write  $Bl(x,r, s)$, 
instead of $Bl(B(x,r), s)$, and likewise for uncentered blossoms. In the latter case, we allow  $r=0$: 
\begin{equation} \label{pointublossom}
Blu(x, 0, s):= \cup\{B(y, s): x\in B(y, s)\}.
 \end{equation}
 We say that $\mu$ {\em blossoms boundedly} if there exists a $K\ge 1$ such that
 for all $r>0 $ and all $x\in X$, $\mu (Blu(x, r, r)) \le K\mu(B(x,r)) < \infty$. 
\end{definition}

\begin{remark} Note that $Bl(S, s)$ is just the $s$-neighborhood of $S$, and the uncentered
blossom is just an abbreviation for the iterated blossom: $Blu(S, s) = Bl (Bl (S, s), s)$.
The introduction of the new notation
 is motivated by the fact that often we will be
blossoming balls $B(x, r)$, and it will be convenient to keep track both of $r$ and $s$.  
\end{remark}

\begin{remark} Blossoms can be defined using closed instead of open balls, in
an entirely analogous way. We mention that in the euclidean case, and more generally, in spaces with the approximate
midpoint property (cf. Definition \ref{amp}) there is no 
difference between $Bl(x,r, r) $ and $B(x, 2r)$, nor between $Blu(x,r, r) $ and $B(x, 3r)$, see \cite[Theorem 2.13]{A2}.
But in general,  balls  are larger. 
\end{remark}

Geometrically doubling spaces have received many names. In
the book \cite{CoWe1}, these spaces are called (in french) spaces of homogenous
type. However, in the paper \cite{CoWe2}, the authors switched notation and started
calling spaces of homogenous
type to those endowed with a doubling measure, after which this became the more common
terminology. Both kinds of spaces (geometrically doubling and with a doubling measure)
have also been called ``doubling spaces", which is why I am avoiding this expression.

\begin{definition} \label{geomdoub} A metric space is {\it $D$-geometrically doubling}  if there exists a positive
integer $D$ such that every ball of radius $r$ can be covered with no more than $D$ balls
of radius $r/2$. 
\end{definition}

Of course,  metric spaces endowed with a doubling measure, 
and geometrically doubling spaces, are closely related, the latter being a generalization
of the former. If $ \mu$ on $X$ is doubling,
then $X$ is geometrically doubling, cf. \cite[Remarque, p. 67]{CoWe1} (but this is
not necessarily the case for measures that blossom boundedly, see Theorem \ref{bb} below).
For a trivial  example of a geometrically doubling space which does not
carry any doubling measure, just consider $\mathbb{Q}$. For a less
trivial example, there are open subsets of $\mathbb{R}$ which do not carry any
doubling measures (cf. \cite[Remark 13.20 (d)]{He}). But if the geometrically doubling space
is complete, then a doubling measure can be defined on it (cf. \cite{LuSa}).

Arguments in analysis that rely on covering theorems often extract a disjoint
collection from the original cover, in such a way that not too much measure
is disregarded.  Now the doubling condition gives us control on the size of all
balls contained in $B(x, 2r)$, regardless of whether they intersect $B(x,r)$ or not.
Since to disjointify we only need to consider balls that do intersect $B(x,r)$, it is advantageous to use
$Bl(x,r,r)$ and $Blu(x, r,r)$ instead of $B(x, 2r)$ and $B(x, 3r)$. The idea of using blossoms can be found in \cite{Li},
for locally compact amenable groups (where in principle there are no balls); 
and  in the metric setting,  it comes from   \cite{NaTa}.
Next  we rewrite,   for the
reader's convenience, a well known argument, using the terminology of blossoms. 

Given a ball $B(x_i, r_i)$ we  shall sometimes use $B_i$ as an abbreviation. It is understood
not only that $B_i= B(x_i, r_i)$ as a set, but also that $x_i$ and $r_i$ are a selected center
and a selected radius of $B_i$ (recall that in arbitrary metric spaces, neither the center nor
the radius of a ball are in general unique). In particular, it might happen that $B_i = B_j$ for some pair $i\ne j$. 

When a measure blossoms boundedly, the
following version of the Vitali covering lemma holds.

\begin{theorem}\label{vitali} {\bf Vitali covering lemma.} Let $(X, d, \mu)$ be a metric measure space.
Assume there exists a constant $K \ge 1$ such that for every $x\in X$ and every
 $r > 0$, $\mu(Blu(x,r, r))\le K \ \mu(B(x,r))$. 
 Then,  for every finite collection of balls $B(x_1, r_1), \dots , B(x_n, r_n)$,
there exists a disjoint subcollection $B(x_{i_1}, r_{i_1}), \dots , B(x_{i_m}, r_{i_m})$
with
\begin{equation}\label{vitali2} 
\mu (\cup_{i=1}^n  B(x_i, r_i)) \le K \  \mu (\cup_{j=1}^m  B(x_{i_j}, r_{i_j})).
 \end{equation}
 \end{theorem}

\begin{proof} We may assume that the original collection, which we will abbreviate by
$B_1, \dots, B_n$, is ordered by decreasing radii. Let $B_{i_1} := B_1$, select $B_{i_2}$
to be the first ball in the list not intersecting $B_{i_1}$, and in general, choose 
$B_{i_k}$ as the first ball in the list not intersecting any of the previously selected balls.
Then the process finishes after a finite number of steps with, let's say, the ball $B_{i_m}$.

We need to control the mass lost with the balls not chosen. Let 
$B_{j_1}^1, \dots,   B_{i_k}^1$ be the collection of all balls intersecting $B_{i_1}$.
Then $\cup_{s=1}^k  B^1_{i_s} \subset Blu(x_1, r_1, r_1)$, so 
$\mu(\cup_{s=1}^k  B^1_{i_s} )\le \mu(Blu(x_1,r_1, r_1))\le K \ \mu(B(x_1,r_1))$. Repeating this
argument with the other balls, we obtain (\ref{vitali2}). 
\end{proof}

The preceding theorem entails the weak type (1,1) of the maximal
operator, and by interpolation,
the corresponding strong  type bounds. One also has the Lebesgue Differentiation
Theorem for measures that blossom boundedly, as in the case of doubling measures.

Regarding the strong type bounds, we mention that once an averaging or maximal
averaging operator is bounded in $L^r$ for some $r >0$, it is bounded in 
$L^p$ for all $p > r$, with
operator norm that approaches 1  (something that is not always
observed in published results). There is no need to use Riesz-Thorin (for
positive sublinear operators) to obtain this, it immediately follows from Jensen's inequality.

\begin{theorem}\label{rieszbds} Let $(X, \mathcal{A}, \mu)$ be a 
measure space, let $0 < r <\infty$, and let 
$T$ be an averaging or maximal averaging operator, bounded on $L^r(X,\mu)$ and with
operator norm $c_r$.
Then
$T$ is bounded on $L^p$ for all $p \ge r$, with operator norm
$c_p\le c_r^{r/p}$. 
\end{theorem}

\begin{proof} 
For all $p$ such that $r < p < \infty$, and all
$f\in L^p$, $f\ge 0$, we have $f^{p/r}\in L^r$, so by Jensen's inequality,
$$
\int [T( f)]^p d\mu \le \int [T( f^{p/r})]^r d\mu \le c_r^r \int f^{p} d\mu,
$$
and  the result follows by taking  $p$-th roots.
\end{proof} 

\section{Local comparability} 

\begin{definition} \label{loccomp}  We say that a measure $\mu$ satisfies a {\it local comparability condition for the radius $r$} 
if there
exists a constant $C\in[1, \infty)$ such that for  all pairs of points $x,y\in X$,
whenever $d(x,y) < r$, we have 
$$\mu(B(x,r))\le C \mu(B(y,r)).$$
If the constant $C$ can be chosen to be independent of $r$, then we say that
 $\mu$ satisfies a {\it $C$ local comparability condition}.
 We denote the
smallest such $C$ by $C(\mu)$ or $C_\mu$.
\end{definition}

Interchanging $x$ and $y$ in the preceding definition leads to 
$$
\frac{1}{C} 
\le
\frac{\mu(B(x,r))}{\mu(B(y,r))}  \le C,
$$
provided $\mu(B(x,r)) >0$ (of course, $\mu(B(x,r)) = 0$ if and only if $\mu(B(y,r)) = 0$).   While  
it is always
possible to assume that $\mu$ has full support, 
by disregarding, if needed, a measure zero set, this can lead to substantial changes
in the geometry of the resulting space, since many properties are not inherited by subsets.
So even though we will always suppose that $\mu$ is not identically $0$, full support will not be assumed.

\begin{example} \label{ultra}Suppose $(X,d)$ is an ultrametric space (so the triangle inequality
is replaced by the stronger condition $d(x,y) \le \max\{d(x,z), d(z,y)\}$). It follows that
$B(x,r) = B(y, r)$
whenever
$d(x,y) < r$, so for every measure $\mu$ on $X$, the 
local comparability condition is trivially satisfied, with $C(\mu)=1$.
\end{example} 

In order to use uncentered blossoms, the following obvious estimate is useful.

\begin{lemma} \label{c2} Let $(X, d, \mu)$ be a metric measure space, and
let $\mu$ satisfy a $C$ 
local comparability
 condition. If $B(x,r)\cap B(y,r)\ne \emptyset$, then 
 $\mu(B(x,r))\le C_\mu^2 \ \mu(B(y,r))$.
\end{lemma}

\begin{proof} Let $z\in B(x,r)\cap B(y,r)$. Since $d(x,z)< r$ and $d(z,y) < r$, we have that
$$
\mu(B(x,r))  \le C_\mu \  \mu(B(z,r)) \le C_\mu^2 \ \mu(B(y,r)).
$$ 
\end{proof}

The local comparability condition can be equivalently stated in terms of closed balls.

\begin{lemma} \label{closed} A measure $\mu$ satisfies a  $C$ local comparability condition
if and only if 
 for every $r>0$ and all pairs of points $x,y\in X$,
whenever $d(x,y) \le r$, we have 
$$\mu(B^{cl}(x,r))\le C \mu(B^{cl}(y,r)).$$
\end{lemma}

\begin{proof} Suppose  $\mu$ satisfies a  $C$ local comparability condition,
and let $d(x,y) \le r$. Then for every $n\ge 1$, $\mu(B(x,r + n^{-1}))\le C \mu(B(y,r+ n^{-1})).$
Taking the limit as $n\to\infty$ we obtain $\mu(B^{cl}(x,r))\le C \mu(B^{cl}(y,r)).$ 
For the other direction, suppose that $0 < d(x,y) < r$, select $N>>1$ such that 
$d(x,y) \le r - N^{-1}$, and use $\mu(B^{cl}(x,r - n^{-1}))\le C \mu(B^{cl}(y,r-n^{-1}))$ 
whenever $n\ge N$. Letting $n\to\infty$ we obtain $\mu(B(x,r ))\le C \mu(B(y,r)).$
\end{proof}

It is easy to see  (and we prove it below) that in geometrically doubling spaces 
local comparability implies boundedness of blossoms. Thus, it is interesting to see
how close, or how far away, is local comparability from doubling. 

But before we
do so, we consider the behavior of averaging operators  when local comparability is missing, with special emphasis  in the case of the standard Gaussian measure.

\section{Averaging operators without local comparability}

As noted in \cite[p. 737]{NaTa}, if $\mu$ satisfies a  $C$ local comparability condition, then,
a simple application of Fubini's Theorem yields, for all averaging operators $A_r$, $r > 0$,
 the uniform bound $\|A_r\|_{L^1\to L^1} \le C$. 
We recall the argument: Suppose that for a fixed radius $s$, and all $x, y \in X$ with $0 < d(x,y) < s$,  we have
$\mu(B(x,s ))\le C \mu(B(y,s)).$ If $0\le f \in L^1(\mu)$, then
\begin{equation}\label{fubini}
\|A_s f\|_{L^1} 
=
\int_X\int_X  \frac{\mathbf{1}_{B(x,s)}(y) }{\mu B(x,s)} f (y) \  d\mu(y) \  d\mu(x)
\end{equation}
\begin{equation}\label{fubini2}
= 
\int_X  f (y) \int_X  \frac{\mathbf{1}_{B(y,s)}(x)}{\mu B(x,s)}  \  d\mu(x) \  d\mu(y)
\le
C \int_X  f (y)  \  d\mu(y),
\end{equation}
so $\|A_s\|_{L^1\to L^1} \le C$, and hence local comparability for $s$ entails the bound $C$ for the corresponding
averaging operator, while local comparability (for all radii) entails uniform bounds for all the averaging operators.

\begin{example}\label{broom}  The  bound $\|A_s\|_{L^p\to L^p} \le C$, for a specific radius $s$,
may fail for all $p\in [1, \infty)$, if $\mu$ lacks local comparability for balls of radius $s$.
We define next a   path connected subset   $B\subset \mathbb{R}^2$, on which we use  the path metric  $d$
instead of the ambient space metric: The distance between
two points is the length of the shortest path joining them. 
Start with the positive $x$-axis. For each $n\ge 1$, select $n$ points from the circumference of radius 1 centered at
$(3 n, 0)$: $z_{n,1}\dots, z_{n,n}\in \{(x,y) \in\mathbb{R}^2| (x - 3n)^2 + y^2 = 1\}$. Join the points 
 $z_{n,1}\dots, z_{n,n}$ to the center $(3 n, 0)$ using straight line segments (radii), let $B$ be the union of 
 the positive $x$-axis with all the ``spikes" attached to the centers $(3 n, 0)$, and let $\mu$ be the counting
measure on the points $(3 n, 0)$ and  $z_{n,1}\dots, z_{n,n}$, for every $n\ge 1$
(clearly, $B$ is separable, and all balls have finite measure). Now $d(z_{n,k}, (3 n, 0)) = 1$, while
if $m\ne n$,  $d(z_{n,k}, z_{m,j}) \ge 5$. Thus,  $\mu B(z_{n,k}, 3/2) = 2$ and $\mu B((3 n, 0), 3/2) = n + 1$,
so local comparability fails for $s = 3/2$. Let $f_n = \mathbf{1}_{\{(3 n, 0)\}}$, and  fix $n, p \ge 1$ with $n >> 
2^p$.
Then  $(A_{3/2} \mathbf{1}_{\{(3 n, 0) \}}) (z_{n, k}) = 1/2$, from whence it follows that
 $\|A_{3/2}\|_{L^p\to L^p}^p \ge n 2^{- p}$. For the same reason (or by interpolation) $A_{3/2}$ satisfies  no  weak (p,p) type bounds. 
 \end{example}

It is nevertheless possible to have uniform $L^p$  bounds for $A_r$  without local comparability. For
instance, in $\mathbb{R}^d$, because of the Besicovitch covering theorem, the 
centered maximal function defined by an arbritraty measure is of weak type (1,1),
and by interpolation, bounded on $L^p$ for all $1 < p <\infty$ (with bounds that
grow exponentially with the dimension).
Thus,  averaging operators satisfy $L^p$ bounds independent of $r$, for $p > 1$. 
 We shall see
later that exponential growth with $d$ can actually happen for $A_s$ with suitably
chosen $s$.

Regarding $L^1$ bounds,   it may happen, even in the absence of local comparability, 
  that the term
$$
 \int_X  \frac{\mathbf{1}_{B(y,s)}(x)}{\mu B(x,s)}  \  d\mu(x)
$$
 in the
left hand side of (\ref{fubini2})
 can still be controled, 
 if the ratio $ \frac{\mu B(y,s)}{\mu B(x,s)} $ becomes
large on sets of sufficiently small measure. 
Next we consider the standard exponential distribution
in one dimension, and the standard gaussian measure in all dimensions. By way of comparison
with the results that follow,
 recall that for Lebesgue measure in $\mathbb{R}^d$, for every 
 $r > 0$ and  $d \ge 1$, $ \|A_r\|_{L^1\to L^1} = 1$ (using (\ref{fubini})-(\ref{fubini2})), 
$ \|A_r\|_{L^{\infty}\to L^{\infty}} \le 1$ (this is obvious),
and  $\|A_r\|_{L^p\to L^p} \le 1$ for $1 < p < \infty$, by interpolation or by Theorem \ref{rieszbds}.

\begin{theorem}  \label{exponential} Consider $\Omega = (0,\infty)$ with the standard exponential
distribution, given by $d P(t) = e^{-t} dt$. Then $P$ satisfies a local comparability condition for
each radius $r$, with optimal  $C(r) \in [e^r, \max\{2, e^r\}]$; thus, it fails to satisfy a local comparability condition. 
However, the averaging operators are uniformly bounded, with 
$1.27 < \|A_1\|_{L^1\to L^1} \le \sup_{r>0} \|A_r\|_{L^1\to L^1} \le 2$.
 \end{theorem} 

\begin{proof} For convenience we extend $P$ to $\mathbb{R}$ by setting
$P(-\infty,0] = 0$.  Fix $r >0$.  First we check that for every $x,y > 0$ such that $|x - y| \le r$, 
$ P((x-r, x+ r)) \le  \max\{2, e^r\} P((y-r, y+ r))$, and the bound $ \max\{2, e^r\}$ cannot be lowered.
We may assume, without loss of generality, that $x < y$, which leads to the consideration
of the following three cases: $x \ge r$,  $x < r \le y$, and  $y < r$.
In the first case, $ P((x-r, x+ r)) \ge P((y-r, y+ r))$, and a computation shows that 
$ P((x-r, x+ r)) / P((y-r, y+ r)) \le e^r$,  with equality when $y = x + r$. 
In the third case, $ P((0, x+ r)) \le P((0, y+ r))$, and $ P( (0, y+ r)) / P((0, x+ r)) \le  (1 - e^{-2r})/(1 - e^{-r}) \le 2$,
since for any decreasing function $h\ge 0$, we have $\int_0^{2 r} h \le 2 \int_0^{r} h$, and furthermore
 $\lim_{r\to 0}  (1 - e^{-2r})/(1 - e^{-r}) =2$, so we get arbitrarily close to 2 by letting $x\to 0, y\to r$, 
 and then $r\to 0$. 
Finally, when  $x < r \le y$,  we have $ e^{-r} = (1 - e^{-2r})/(e^r - e^{-r}) \le  P( (y-r, y+ r)) / P((x-r, x+ r)) \le  (1 - e^{-2r})/(1 - e^{-r}) \le 2$.

Next, note that for every $w\in (0,\infty)$ and every $r >0$,
$P(B(w, r)) =\int_{\max\{0, w -r\}}^{w + r} e^{-t} dt$, so if $w\ge r$, by the convexity of $e^{-t}$ we have
\begin{equation}\label{convex}
P(B(w, r)) \ge e^{-w } 2r,
\end{equation} 
 while if  $w< r$, then $P(B(w, r)) =\int_{0}^{w + r} e^{-t} dt = 1 -  e^{-r -w} \ge 1 -  e^{-r} $.
In order to bound 
\begin{equation}
\int_{0}^{\infty}  f (y) \int_{0}^{\infty} \frac{\mathbf{1}_{B(y,r)}(x)}{P( B(x,r))}  \  d P(x) \  d P(y),
\end{equation} 
we break up the outer integral into $\int_{0}^{\infty} = \int_{0}^{2 r} + \int_{2r}^{\infty}.$
On the region where $y \ge 2r$, since $|x - y| \le r$,  both $B(x,r) = (x - r, x + r)$ and $B(y,r) = (y - r, y + r)$ 
so 
$$
 \int_{0}^{\infty}  \frac{\mathbf{1}_{B(y,r)}(x)}{P( B(x,r))}  \  d P(x)
=
\int_{y-r}^{y + r}  \frac{e^{-x}}{P((x -r, x + r))}  \  d x
\le
\int_{y-r}^{y + r}  \frac{dx }{2 r}  \  d x
= 1.
$$
If $0 < y < 2r$,  
$$
 \int_{0}^{\infty}  \frac{\mathbf{1}_{B(y,r)}(x)}{P( B(x,r))}  \  d P(x)
=
\int_{\max\{0, y -r\}}^{y + r}  \frac{e^{-x}}{P((x -r, x + r))}  \  d x
=
\int_{\max\{0, y -r\}}^{ r}  + \int_{r}^{y + r} .
$$
Now
$$
\int_{\max\{0, y -r\}}^{ r}  \frac{e^{-x}}{P((x -r, x + r))}  \  d x 
\le
\int_{\max\{0, y -r\}}^{ r}  \frac{e^{-x}}{1 - e^{-r}}  \  d x 
=
 \frac{e^{- \max\{0, y -r\}} - e^{-r}}{1 - e^{-r}} \le  \frac{1 - e^{-r}}{1 - e^{-r}} = 1,
$$
while
 $$
\int_{r}^{y + r}  \frac{e^{-x}}{P((x -r, x + r))}  \  d x
\le
\int_{r}^{y + r}  \frac{dx }{2 r}  \  d x
\le 1,
$$
by (\ref{convex}).
Hence,  
\begin{equation}
\int_{0}^{\infty}  f (y) \int_{0}^{\infty} \frac{\mathbf{1}_{B(y,r)}(x)}{P( B(x,r))}  \  d P(x) \  d P(y)
=
2 
\int_{0}^{2r}  f (y)  \  d P(y) +  \int_{2r }^{\infty}  f (y)  \  d P(y) \le 2\|f\|_1.
\end{equation} 
Next, take $r=1$. We show  that $1.27 < \|A_1\|_{L^1\to L^1}$.
Recall that
 when the centered maximal operator acts on a measure $\nu$, it is defined via 
$$
M_{\mu} \nu (x) := \sup _{\{r > 0: \mu (B(x, r)) > 0\}} \frac{\nu (B(x, r))}{\mu (B(x, r))}.
$$
By a standard approximation argument, instead of a function, we consider
a Dirac delta placed at 1. 
 If $x\in(0,1)$, then
$$
A_1\delta_1(x) = \frac{1}{1 - e^{-1-x}},
$$
while if 
 $x\in [1, 2)$, then
$$
A_1\delta_1(x) = \frac{e^x}{e - e^{-1}}.
$$
Using the change of variables $u=e^x$ and  integrating explicitly, we obtain
$$
\|A_1\|_{L^1(P)\to L^1(P)} \ge \|A_1 \delta_1\|_{L^1(P)} = \int_0^2 A_1\delta_1(x) \ e^{-x} \ dx =
\int_0^1 \frac{dx}{e^x - e^{-1}} + \int_1^2 \frac{dx}{e - e^{-1}} 
$$
$$
=  e \log\left( \frac{e -  e^{-1}}{1 - e^{-1}}\right) - e + \frac{1}{e - e^{-1}} > 1.27.
$$
\end{proof}

Next we consider the case of the standard Gaussian measure $\gamma_d$
in $\mathbb{R}^d$. Here local comparability fails for every single radius $r >0$,
but nevertheless, the averaging operators $A_r$ defined by $\gamma_d$ satisfy 
uniform $L^1$-bounds (in $r$). 
We use $\|x\|_2 :=\left( x_1^2+ x_2^2+\dots + x_d^2\right)^{1/2}$ to denote the euclidean distance
in $\mathbb{R}^d$. Recall that the standard gaussian measure  is given by
$$
d \gamma^d (x) = \frac{e^{ - \frac{\|x\|_2 ^2}{2}}}{(2\pi )^{d/2}} \  dx.
$$

\begin{theorem}  \label{gaussian} Let $(\Omega, d, P)$ be $(\mathbb{R}^d, \|x\|_2,
\gamma^d)$, where $\gamma^d$ is the 
 standard gaussian measure. Given any $r > 0$,  $\gamma^d$ does not 
 satisfy a local comparability condition for $r$. However, for all $p \ge 1$, 
$$
\sup_{r>0}\|A_r\|_{L^p\to L^p} \le \left( 2^{d -1}\sqrt{2 \pi d}  + 
 \sqrt{\pi (d + 1)}\left(  \frac{2}{\sqrt3}\right)^{d + 1}\right)^{1/p}.
$$
Thus, for  every $\varepsilon > 0$  and every $d$ large enough, 
$$
\sup_{r>0}\|A_r\|_{L^p\to L^p} \le \left( 2+ \varepsilon \right)^{d/p}.
$$
Furthermore,
for  every $p$ in $[1,\infty)$ and every $d$ sufficiently large, the weak type $(p,p)$
constants satisfy $\left\|A_{\frac{\sqrt{3d - 3}}{2}} \right\|_{L^p\to L^{p, \infty}} >1.019^{d/p}$. 
 \end{theorem} 
 
 Of course, for $p > 1$ the fact that the operators $A_r$ are uniformly bounded in $L^p$
 follows from the corresponding bounds for the maximal operators. But  the boundedness of $A_r$ for $p=1$ in the preceding result is new, and
 for $p$ close to 1,  the  bounds obtained by interpolation  do not
 blow up when $p \to 1$, unlike the case of the maximal function inequalities.

The next obvious lemma tells us that the normalizing constants $(2\pi )^{-d/2}$ in the probabilities
can be omitted from certain formulas.

\begin{lemma} Given a measure $\mu$ and a constant $c > 0$, for all $r>0$, all $f\in L^1$, and all
$x\in X$, we have 
$A_{r,\mu} f (x) = A_{r,c \mu} f (x)$, and 
$\|A_{r,\mu} \|_{L^1(\mu) \to L^1 (\mu)} = \|A_{r,c \mu} \|_{L^1(c \mu) \to L^1 (c \mu)}.$
\end{lemma}

For the next lemma it will be more convenient to use closed balls. Of course, from the
viewpoint of the gaussian measure this makes no difference, since the boundaries of balls
have measure zero. Next we introduce some notation. We often omit the center of the 
sphere $\mathbb S^{d-1} (0, r)$ when it is the origin, and if it is the unit sphere ($r = 1$) we also
omit the radius. Given a unit vector $v\in \mathbb
R^{d}$ and $s \in [0, 1)$, the $s$ spherical
cap about $v$ is the set $C(s, v) :=\{\theta \in
\mathbb S^{d-1}: \langle \theta, v\rangle \ge s\}$.
Spherical caps are just geodesic balls $B^{cl}_{\mathbb
S^{d-1}}(x, r)$ in  $\mathbb S^{d-1}$.  Let  $e_1 =
(1,0,\dots , 0)$.  Given any angle $r\in (0, \pi/2)$,
writing $s = \cos r$, we have
\begin{equation}\label{cap1}
B^{cl}_{\mathbb S^{d-1}}(e_1, r) =
C(s, e_1).
\end{equation}
In particular,  $B^{cl}_{\mathbb S^{d-1}}(e_1, \pi/6) =
C(\sqrt3/2, e_1)$.

\begin{lemma} \label{firstop} Let $\mu$  be a rotationally invariant measure on $\mathbb{R}^d$, let $r>0$, 
and let $\|v\|_2= r$. Then $\mu (B^{cl}(0, r)) \le 2^{d -1}\sqrt{2 \pi d} \ \mu (B^{cl}(v, r))$.
\end{lemma}

\begin{proof} First we recall a well known volumetric argument giving upper bounds on the size
of $r$-nets on the $r$-sphere. Since for this part we only use the normalized area 
$\sigma^{d-1}_N$
 on the sphere $\mathbb S^{d-1} (0, r)$,
we can take $r=1$. Let $\{v_1, \dots, v_M\}$ be a maximal set of unit vectors in 
$\mathbb{R}^d$,
subject to the condition that for $i\ne j$, $\|v_i - v_j\|_2 \ge 1$. The maximality of $\{v_1, \dots, v_M\}$  entails that
 $B^{cl}( 0, 1) \subset \cup_1^M B^{cl}(v_i, 1)$, as the following argument shows. Suppose $y\in B^{cl}( 0, 1)
\setminus \cup_1^M B^{cl}(v_i, 1)$. Then $y\ne 0$, because the origin belongs to all the balls under consideration,
so $v:= y/\|y\|_2$ lies on the unit sphere. Now if $v\in B^{cl}(v_i, 1)$  for some index $i$, then
 $y\in B^{cl}(v_i, 1)$ by the  convexity of the ball, so $v\in B^{cl}( 0, 1)
\setminus \cup_1^M B^{cl}(v_i, 1)$. But then  $\|v_i - v\|_2 > 1$ for $i= 1,\dots, M$, contradicting the maximality of  $\{v_1, \dots, v_M\}$.

 Next, note that  all balls $ B^{cl}(v_i, 1/2)$
have disjoint interiors, so their radial projections (from the origin) into the unit sphere 
$\mathbb S^{d-1}$ also have disjoint interiors. By rotational invariance we may assume that
$v_1 = e_1$. Now the tangent lines to $ B^{cl}(e_i, 1/2)$ starting from the origin, form
an angle of $\pi/6$ with $e_1$, since the radii of $ B^{cl}(e_i, 1/2)$ are perpendicular to these
tangent lines, and $\sin (\pi/6) = 1/2$.  Thus, the radial projection  of $ B^{cl}(e_i, 1/2)$
into $\mathbb S^{d-1}$ is the geodesic ball, or spherical cap, $B^{cl}_{\mathbb S^{d-1}}(e_1, \pi/6) =
C(\sqrt3/2, e_1)$. By \cite[Lemma 2.1]{AlPe}, 
\begin{equation}\label{caps}
\frac{ 1}{ 2^{d-1}\sqrt{2 \pi d}}
 \le
\sigma^{d-1}_N (C(\sqrt3/2, e_1)),
\end{equation}
so  $\sum_1^M \sigma^{d-1}_N (C(\sqrt3/2, v_i)) 
=
M \sigma^{d-1}_N (C(\sqrt3/2, e_1)) \le 1$, 
from which it follows that $M\le  2^{d-1}\sqrt{2 \pi d}$.

Let us now return to the original $\mu$ and  $r > 0$.
 By invariance under rotations, for all $i$ we have $\mu  B^{cl}(r v_i, r) = \mu  B^{cl}(r v_1, r)$. 
Since  $ B^{cl}(0, r) \subset \cup_1^M B^{cl}(r v_i, r)$, it follows that
$\mu  B^{cl}(0, r) \le \mu \cup_1^M B^{cl}(r v_i, r)  \le  \sum_1^M \mu B^{cl}(r v_i, r)  \le  2^{d-1}\sqrt{2 \pi d} \ \mu (B^{cl}(r v_1, r))$.
\end{proof}

\begin{corollary} \label{central} Let $\mu$  be a radial, radially decreasing measure  on $\mathbb{R}^d$. 
That is, $d \mu (x) = f(x) dx$, where
$f$ is a locally integrable, radial and radially decreasing function. Then for every $g\in L^1(\mu)$,
$\|\mathbf{1}_{B(0,r)} A_{r,\mu} g \|_{L^1(\mu)} \le   2^{d -1}\sqrt{2 \pi d} \  \|\mathbf{1}_{B(0, 2r)} g \|_{L^1(\mu)}$,
where $ A_{r,\mu}$ is defined using the closed balls  $ B^{cl}(x, r)$.
\end{corollary}

\begin{proof} Note that of all balls  $ B^{cl}(x, r)$ containing the origin, the ones with smallest measure
are those for which $\|x\|_2 = r$, since for every $t\in [0,1)$, for all $z\in  B^{cl}(x, r) \setminus  B^{cl}(t x,  r)$,
and all  $y\in  B^{cl}(t x, r) \setminus  B^{cl}(x,  r)$, $f(z ) < f(y)$. 
Now if $x\in B(0,r) \setminus \{0\}$ and $0\le g \in L^1(\mu)$, then 
$$
 A_{r,\mu} g (x) \le \frac{ \|\mathbf{1}_{B(0, 2r)} g \|_{L^1(\mu)}}{\mu  B^{cl}(r x/\|x\|_2,  r)},
$$
so the result follows from the previous lemma.
\end{proof}

For the rest of this section, we use $\mu$ to denote the non-normalized gaussian 
$
d \mu (x) = e^{ - \frac{\|x\|_2 ^2}{2}}.$
Occasionally we will  write $\mu^d$ to specify the dimension $d$.

\begin{lemma} \label{secondopx} Let $r>0$ and let $x\in  \mathbb{R}^d  \setminus B(0, r)$. 
Then  
$$\mu   B (x, r) \ge   \frac{e^{ - \frac{\|x\|_2 ^2}{2}}  \lambda^d B(0, r)}{\sqrt{\pi (d + 1)}}  \left(  \frac{\sqrt3}{2}\right)^{d + 1}.$$
\end{lemma}

\begin{proof}  When needed, we will distinguish between balls in $\mathbb{R}^{d}$ and 
$\mathbb{R}^{d-1}$ by writing $B^d$ and $B^{d-1}$ respectively. 
Using rotational invariance, we may assume that $x = s e_1$, with
$s \ge r$. Since 
$$
\mu   B (s e_1, r) \ge \mu  \left( B (s e_1, r) \cap B (0, s)\right) \ge e^{-s^2/2}
\lambda^d \left( B (s e_1, r) \cap B (0, s)\right),
$$
all we need to do is to show that
\begin{equation} \label{cap}
\frac{\lambda^d \left( B (s e_1, r) \cap B (0, s)\right)}{\lambda^d \left(B (0, r)\right)} 
\ge   \frac{1}{\sqrt{\pi (d + 1)}}\left(  \frac{\sqrt3}{2}\right)^{d + 1}.
\end{equation}
First of all, note that the ratio in the left hand side of the preceding inequality is minimized
when $s = r$, so we suppose this is the case. Second, dividing all radii by $s$ and cancelling
the factors $s^d$, we may take
$s = 1$. Now 
\begin{equation*}
 \lambda^{d} (B(0,1) \cap B( e_1, 1)) = 2 \lambda^d (B^d  (0,1) \cap \{x_1 \ge
2^{-1}\}) 
\end{equation*}
\begin{equation*}
= 
2 \lambda^{d-1} (B^{d-1}(0,1))\int_{1/2}^1\left(\sqrt{1-x_1^2}\right)^{d-1}dx_1
\end{equation*}
\begin{equation*}
 \ge 
 2 \lambda^{d-1} (B^{d-1} (0,1)) \int_{\pi/6}^{\pi/2}\cos^d t \sin t dt 
 =   \frac{ 2}{ d + 1}\left( \frac{ \sqrt 3}{2}\right)^{d+1} \lambda^{d-1} (B^{d-1}(0,1)).
 \end{equation*}
 
Using $\lambda^d (B^{d}(0,1)) = 
\frac{\pi^{d/2}}{\Gamma (1 + d/2)} $, together with the following Gamma function
estimate (consequence of the log-convexity of
$\Gamma$ on $(0,\infty)$, cf. Exercise 5, pg. 216 of \cite{Web})
\begin{equation}\label{ratio}
\left(\frac{d}{2}\right)^{1/2}
\le
\frac{\Gamma (1 + d/2)}{\Gamma (1/2 + d/2)},
\end{equation}
we get
\begin{equation} \label{cap1}
\frac{\lambda^d \left( B (e_1, 1) \cap B (0, 1)\right)}{\lambda^d \left(B (0, 1)\right)} 
\ge  \frac{1}{\sqrt{\pi (d + 1)}}\left(  \frac{\sqrt3}{2}\right)^{d + 1}.
\end{equation}\end{proof}

{\em Proof of Theorem \ref{gaussian}.}   Fix $r >0$ and take $\|x\|_2 >>1 + r$.  To see that $\mu$ does not satisfy a local comparability condition for $r$, we consider the balls
$B(x,r)$ and $B( (1 + 3 r/\|x\|_2)  x,r)$. Since their centers are at distance $3r$, they are disjoint.
However, $B(x,r)$, $B( (1 + 3 r/(2\|x\|_2))  x,r)$ and $B( (1 + 3 r/\|x\|_2)  x,r)$ form an intersecting chain of balls of length 3, so applying Lemma \ref{c2} twice, local comparability for $r$ would imply
that the measures of $B(x,r)$ and $B( (1 + 3 r/\|x\|_2)  x,r)$ are comparable, for every $x$.
However, 
$$
\lim_{x\to\infty} \frac{\mu(B ((1 + 3 r/\|x\|_2) x,  r))}{\mu(B (x,  r))} \le 
\lim_{x\to\infty} \frac{e^{ - \frac{(\|x\|_2 + 2 r)^2}{2}}
\lambda^{d} B(0,r)}{\  e^{ - \frac{(\|x\|_2 +  r)^2}{2}} \lambda^{d} B(0,r)} = 0.
$$

In order to prove that $\sup_{r>0}\|A_r\|_{L^1\to L^1} \le (2 + \varepsilon )^d$ for $d$ large, we split $A_{r,\mu}$ into 
$A_{r,\mu} = \mathbf{1}_{B(0,r)} A_{r,\mu} +
\mathbf{1}_{B(0,r)^c} A_{r,\mu}$. Let $0 \le g \in L^1(\mu)$.  The bound 
$$
\|\mathbf{1}_{B(0,r)} A_{r,\mu} g \|_{L^1(\mu)} \le   2^{d -1}\sqrt{2 \pi d} \  \|\mathbf{1}_{B(0, 2r)} g \|_{L^1(\mu)}
$$
is a special case of Corollary \ref{central}, together with the fact that open and closed balls have
the same gaussian measure.
Regarding the second term,
\begin{equation}\label{secondop}
\|\mathbf{1}_{B(0,r)^c} A_{r,\mu} g \|_{L^1(\mu)} 
= 
\int_{\mathbb{R}^d}\int_{\mathbb{R}^d}\ 
 \frac{\mathbf{1}_{B(x, r)}(y) \mathbf{1}_{B(0,r)^c} (x) }{\mu B(x, r)}  g(y) \  d\mu(y) \  d\mu(x)
 \end{equation}
 \begin{equation}\label{secondop1}
= 
\int_{\mathbb{R}^d}   g (y) \int_{\mathbb{R}^d}  
\frac{\mathbf{1}_{B(y, r)}(x) \mathbf{1}_{B(0,r)^c} (x) }{\mu B(x, r) }\  d\mu(x) \  d\mu(y).
\end{equation}
Now fix $y$. By Lemma \ref{secondopx}
we have
$$
\int_{\mathbb{R}^d}  
\frac{\mathbf{1}_{B(y, r)}(x) \mathbf{1}_{B(0,r)^c} (x) }{\mu B(x, r) }\  d\mu(x) = 
\int_{B(0,r)^c \cap B(y, r)}  
\frac{ e^{ - \frac{\|x\|_2 ^2}{2}}  \  dx}{\mu B(x, r) }
$$
$$
\le
 \sqrt{\pi (d + 1)}\left(  \frac{2}{\sqrt3}\right)^{d + 1}\ \int_{B(y, r)}  
\frac{ dx}{\lambda^d B(0, r) } =  \sqrt{\pi (d + 1)}\left(  \frac{2}{\sqrt3}\right)^{d + 1}.
$$
Therefore
$$
\| A_{r,\mu} g \|_{L^1(\mu)} 
= 
\|\mathbf{1}_{B(0,r)} A_{r,\mu} g \|_{L^1(\mu)} 
+
\|\mathbf{1}_{B(0,r)^c} A_{r,\mu} g \|_{L^1(\mu)} 
$$
$$
\le \left( 2^{d -1}\sqrt{2 \pi d} \ +  \sqrt{\pi (d + 1)}\left(  \frac{2}{\sqrt3}\right)^{d + 1}\right) \|g\|_{L^1(\mu)},
$$
and from Theorem \ref{rieszbds} we conclude that for every $p \ge 1$, 
$$
\| A_{r,\mu} \|_{L^p(\mu) \to L^p(\mu)} 
\le \left( 2^{d -1}\sqrt{2 \pi d} \  +  \sqrt{\pi (d + 1)}\left(  \frac{2}{\sqrt3}\right)^{d + 1}\right)^{1/p}.
$$
Regarding the lower bounds for the weak type constants, the argument we use is the same
as in \cite{CriSjo}, together with gaussian concentration. The basic idea is that since the standard
gaussian measure in $\mathbb{R}^{d}$ behaves essentially as normalized area on the sphere 
$\mathbb{S}^{d-1}(\sqrt{d})$, centered at 0 and 
of radius $\sqrt{d}$, a single well chosen radius is enough to witness the exponential growth of
constants with the dimension. So the argument given by Criado and Sjogr\"en for the
maximal operator essentially yields the same result for certain averaging operators. 
Of course, since the standard gaussian measure is not singular, but absolutely continuous, 
one still has to show that small changes in the center of a ball lead to small
changes in the average.

\vskip .2 cm

First, we estimate from below the measure of the region bounded between 
$\mathbb{S}^{d-1}(\sqrt{d-1} - 1/\sqrt{d-1})$ and $\mathbb{S}^{d-1}(\sqrt{d - 1})$
(the region between the radii $\sqrt{d}$ and $\sqrt{d-1} $ could be used in an entirely analogous way, but since we
want to  cite estimates from \cite{CriSjo} directly, rather than to redo them, and they use $\sqrt{d-1} $
as the largest radius,
so do we). 
Denote by 
$$
\sigma_{d-1}(\mathbb{S}^{d-1})
:= 
\frac{d \  \pi^{d/2}}{\Gamma (1 + d/2)}
$$ 
the area of the
unit sphere $\mathbb{S}^{d-1}$.   In what follows
 we utilize Stirling's formula, as well as
approximations to $e$, so the inequality below holds for $d$ large enough.
Since $g(r) := r^{d-1} e^{-r^2/2}$ is increasing for
$0< r < \sqrt{d-1}$, 
\begin{equation}\label{area1}
\int_{B(0, \sqrt{d-1}) \setminus B(0, \sqrt{d-1} - 1/\sqrt{d-1}) } \  d  \gamma^d (x)
=
\frac{\sigma_{d-1}(\mathbb{S}^{d-1})}{(2 \pi )^{d/2}} 
\int_{\sqrt{d-1} - 1/\sqrt{d-1}}^{\sqrt{d-1}} g(r) dr
\end{equation}
\begin{equation}\label{area2}
\ge
\frac{d \  \pi^{d/2}}{(2 \pi )^{d/2} \Gamma (1 + d/2)} g(\sqrt{d-1} - 1/\sqrt{d-1}) 
\int_{\sqrt{d-1} - 1/\sqrt{d-1}}^{\sqrt{d-1}}  dr
> 
\frac{1}{(\pi  e^3 d)^{1/2}}
\end{equation}
for $d$ large.
We want to obtain a good estimate from below for 
$$A_{\frac{\sqrt{3d - 3}}{2}} \mathbf{1}_{B(0, \frac{\sqrt{3d - 3}}{2})}
$$
on the region $D:= B(0, \sqrt{d-1}) \setminus B(0, \sqrt{d-1} - 1/\sqrt{d-1}) $.
Let $1 \le p < \infty$, let $0 < r, R < 1$, and let $c  >0$ satisfy, for all $\alpha >0$, 
the inequality
\begin{equation}\label{weaktypepgauss}
\gamma^d (\{A_{R \sqrt{d-1}, \gamma^d} \mathbf{1}_{B(0, r \sqrt{d-1})}  \ge \alpha\}) 
\le \left(\frac{c \|\mathbf{1}_{B(0, r \sqrt{d-1})} \|_{L^p(\gamma^d )}}{\alpha}\right)^p,
\end{equation}
or equivalently
\begin{equation}\label{weaktypepgauss2}
\frac{\alpha  \gamma^d (\{A_{R \sqrt{d-1}, \gamma^d} \mathbf{1}_{B(0, r \sqrt{d-1})}  \ge \alpha\})^{1/p} }{\gamma^d (B(0, r \sqrt{d-1}))^{1/p}}
\le c.
\end{equation}
To find a (sufficiently high) uniform lower bound $\alpha$  for 
\begin{equation}\label{weaktypepgauss3}
A_{R \sqrt{d-1}, \gamma^d}  \mathbf{1}_{B(0, r \sqrt{d-1})} (x) 
= 
\frac{ \gamma^d ( B(0, r \sqrt{d-1})\cap (B(x, R \sqrt{d-1}))}{\gamma^d (B(x, R \sqrt{d-1}))},
\end{equation}
whenever  $x \in D$, we use
the following facts,  taken from \cite[Proof of Lemma 5.1]{CriSjo} (note that we have chosen
a different, more common normalization for the gaussian measure, but this makes no essential difference; for the
justification of the assertions below we refer the reader to the original paper). Criado
and Sj\"ogren show that if $\|x\|_2^2 = \sqrt{d-1}$, then for each $R\in (0,1)$, an $r\in (0,1)$ can be chosen in such a
way that 
\begin{equation}\label{crisjo1}
\frac{ \gamma^d ( B(0, r \sqrt{d-1})\cap B(x, R \sqrt{d-1}))}{\gamma^d (B(x, R \sqrt{d-1}))} 
\ge
\Theta\left(\frac{1}{\sqrt{d-1}}\right),
\end{equation}
where  $\Theta$ denotes exact order. In view of this bound and of the
denominator in (\ref{weaktypepgauss2}),
we want to select $r$ as small as possible, which is where the choice $r = R = \sqrt{3}/2$
comes from, as we shall see next.
Let 
$$
F(t, R) := \left(t - \frac{(1 + t - R^2)^2}{4}\right) e^{-t},
$$
where $(1 - R)^2 \le t \le (1 + R)^2$, 
let 
$$
t(R) := 2 + R^2 - \sqrt{1 + 4 R^2} ,
$$
and let
$$
G(R) := F(t(R), R).
$$
For each fixed $R$, $t(R)$ maximizes $F$, so $F(t, R) \le G(R)$ 
(cf. \cite[p. 609]{CriSjo} for justifications of the choices and claims made). Set $r(R):= \sqrt{t(R)}$.
Since 
 $R = \sqrt3 /2$ is the only zero of $t^\prime$ in $(0,1)$ and 
$t^{\prime\prime}(\sqrt3 /2) > 0$, the function $t(R)$ has a local minimum there, which is easily seen to
be the unique global minimum (for instance, by checking the endpoints).
Hence, for $0 < R < 1$, $t(R) \ge t(\sqrt3 /2) 
= 3/4$, and we choose $r = \sqrt3 /2 =R$. Given $x\in D$, by rotational invariance we may suppose that
$x = u e_1$, where $\sqrt{d-1} - 1/\sqrt{d-1} \le u \le \sqrt{d-1}$, and as before, $e_1$ is the first vector in the
standard basis of $\mathbb{R}^d$. Criado and Sj\"ogren show that 
\begin{equation}\label{crisjo2}
\gamma^d B(\sqrt{d-1} \ e_1, R \sqrt{d-1}))
 \le
\frac{ 2 \sigma_{d-2}(\mathbb{S}^{d-2} )\left(\sqrt{d-1}\right)^{d} R}{(d - 1) \sqrt{1 - R^2}}
\ G(R)^{\frac{d -1}{2}} =: V(R).
\end{equation}

Using this, for $u e_1\in D$ we have
\begin{equation}\label{average1}
\frac{ \gamma^d ( B(0, \frac{\sqrt3}{2} \sqrt{d-1})\cap B(u e_1, \frac{\sqrt3}{2} \sqrt{d-1}))}
{\gamma^d (B(u e_1, \frac{\sqrt3}{2} \sqrt{d-1}))} 
\end{equation}
\begin{equation}\label{average2}
\ge
\frac{ \gamma^d ( B(0, \frac{\sqrt3}{2} \sqrt{d-1})\cap B( \sqrt{d-1} \ e_1, \frac{\sqrt3}{2} \sqrt{d-1}))}
{\gamma^d (B( \sqrt{d-1} \ e_1, (\frac{\sqrt3}{2}  + \frac{1}{d-1}) \sqrt{d-1}))}
\end{equation}
\begin{equation}\label{average3}
\ge
\frac{ \gamma^d ( B(0, \frac{\sqrt3}{2} \sqrt{d-1})\cap B( \sqrt{d-1} \ e_1, \frac{\sqrt3}{2} \sqrt{d-1}))}
{{V(\frac{\sqrt3}{2}  + \frac{1}{d-1})}}
\end{equation}
\begin{equation}\label{average4}
=
\frac{ V(\frac{\sqrt3}{2})}
{V(\frac{\sqrt3}{2}  + \frac{1}{d-1})} 
\frac{ \gamma^d ( B(0, \frac{\sqrt3}{2} \sqrt{d-1})\cap B( \sqrt{d-1} \ e_1, \frac{\sqrt3}{2} \sqrt{d-1}))}
{V(\frac{\sqrt3}{2})}.
\end{equation}
The last factor in the preceding line is  bounded below (cf. \cite[p. 610]{CriSjo})
by 
\begin{equation}\label{crisjo11}
\frac{ c_0}{\sqrt{d-1}},
\end{equation}
where  $c_0 $ is a strictly positive constant 
(in particular, it is independent of $d$; it may depend on the choices
of $R$ and $r$, which are equal to  $\sqrt3/2$ in our case).
Regarding 
$$\frac{ V(\frac{\sqrt3}{2})}
{V(\frac{\sqrt3}{2}  + \frac{1}{d-1})},
$$
as $d$ becomes large, the changes that $R$ and $R^2$ undergo in the fraction of formula (\ref{crisjo2}),
when $R$ takes the value $\frac{\sqrt3}{2}  + \frac{1}{d-1}$ instead of $\frac{\sqrt3}{2}$,
become vanishingly small, so all we need to do is  to bound 
\begin{equation}\label{G}
\left(\frac {G\left(\frac{\sqrt3}{2}\right) }{G\left(\frac{\sqrt3}{2}  + \frac{1}{d-1}\right)}\right)^{\frac{d -1}{2}}
\end{equation}
from below. A computation shows that 
$G^{\prime\prime}(\sqrt3/2) <  0$, so $G$ is locally concave at $\sqrt3/2$, and thus, 
for $d$ sufficiently high, $G(\sqrt3/2 ) + 1/(d-1)) \le G(\sqrt3/2) + G^{\prime}(\sqrt3/2)/(d-1)$. Since $G(\sqrt3/2) = 1/(2 e^{3/4})$ and $G^{\prime}(\sqrt3/2) = \sqrt{3}/(2 e^{3/4})$, we have that
\begin{equation}\label{G1}
\left(\frac {G\left(\frac{\sqrt3}{2}\right) }{G\left(\frac{\sqrt3}{2}  + \frac{1}{d-1}\right)}\right)^{\frac{d -1}{2}}
\ge
\left(\frac{1}{1 + \frac{\sqrt3}{d -1}} \right)^{\frac{d -1}{2}} > e^{-1}
\end{equation}
for $d$ large enough.

Finally, since $g(r) := r^{d-1} e^{-r^2/2}$ is increasing for
$0< r < \sqrt{d-1}$, 
$$
\gamma^d ( B(0, \frac{\sqrt3}{2} \sqrt{d-1})) 
=
\frac{\sigma_{d-1}(\mathbb{S}^{d-1})}{(2 \pi )^{d/2}} 
\int_{0}^{\frac{\sqrt3}{2} \sqrt{d-1}} g(r) \  dr
$$
$$
\le \frac{\sigma_{d-1}(\mathbb{S}^{d-1})}{(2 \pi )^{d/2}} 
g\left(\frac{\sqrt3}{2} \sqrt{d-1}\right)  \int_{0}^{\frac{\sqrt3}{2} \sqrt{d-1}} dr
= \frac{\sigma_{d-1}(\mathbb{S}^{d-1})}{(2 \pi )^{d/2}} 
\left(\frac{\sqrt3}{2} \sqrt{d-1}\right)^d   e^{- \frac{3 d - 3}{8}}.
$$
Using Stirling's formula, for $d$ large the right hand side of the preceding equality  can be bounded above
by
$$
\sqrt{d} 
\left(\frac{3 e^{1/4}}{4}\right)^{\frac{d}{2}}.
$$ 
Putting together in formula (\ref{weaktypepgauss2}) the bounds (\ref{area1})-(\ref{area2}), (\ref{crisjo1})-(\ref{G1}), and the last estimate,
we conclude that for $d$ sufficiently large,
$$
\|A_{\sqrt{3d - 3}/2} \|_{L^p\to L^{p, \infty}} 
\ge
\left(\frac{2}{3^{1/2} e^{1/8}}\right)^{\frac{d}{p}} \Theta\left(\frac{1}{d^{1/2 + 1/p}}\right).
$$
Now $\frac{2}{3^{1/2}  e^{1/8}} > 1.019$, so for $d$ large enough the factor 
$\Theta\left(\frac{1}{d^{1/2 + 1/p}}\right)$ can be absorbed in the exponential, and we get
$$
\left\|A_{\frac{\sqrt{3d - 3}}{2}} \right\|_{L^p\to L^{p, \infty}} >1.019^{d/p}.
$$
\qed

I do not know whether the uniform $L^1$  boundedness of the operators $A_r$ in the two preceding cases
(exponential and gaussian) are instances of a more general result in euclidean spaces, or whether there are measures
$\nu$ in $\mathbb{R}^d$ for which $\sup_{r>0}\|A_{r,\nu}\|_{L^1\to L^1} = \infty$. Curiously, for the one-directional averaging
operators in $\mathbb{R}$, examples of such measures are easy to find.

Given $\mu$ on $\mathbb{R}$  and $s > 0$, define, for all $x \in \mathbb{R}$  such that
 $\mu ([x, x + s]) > 0$,  the right directional averaging operator as
$$
A_{s,\mu}^r f(x) := \frac{1}{\mu ([x, x + s])} 
\int_{\mathbb{R}}  f (y) \mathbf{1}_{[x, x + s]}(y)   \  d \mu (y).
$$

\begin{theorem}  \label{onedir} There exists a measure $\nu$ on $\mathbb{R}$ 
such that  $\|A^r_{1,\nu}\|_{L^1\to L^1} = \infty$.
 \end{theorem} 

 By way of comparison, for $p>1$ and denoting by $M^r$ the right directional Hardy-Littlewood maximal operator,
given any  $\mu $ on $\mathbb{R}$ we have that
 $$
\sup_{s>0} \|A^r_{s,\mu}\|_{L^p\to L^p} \le  \|M^r_{\mu}\|_{L^p\to L^p} 
 =  \frac{p }{p-1}$$ 
(cf. \cite[p. 102, Exercise 2.1.11]{Gra}).

\begin{proof}  Let $B := \cup_{n \ge 0} [2n , 2n + 1)$, and let $d\nu (x)  := (\mathbf{1}_B (x) + e^{-x}) dx$. 
Again by approximation, we can use Dirac deltas instead of functions. For  $x\in [2n, 2 n + 1)$ we have that 
$$
A^r_{1,\nu} \delta_{2n + 1} (x) \ge  \frac{1}{2 n + 1 - x + e^{-2n}}.$$
 Since $d\nu (x)> 1$ on  $[2n, 2n + 1)$,
we conclude that $\lim_{n\to \infty}   \|A^r_{1,\nu} \delta_{2n + 1} \|_{L^1} = \infty$.
\end{proof}

\section{Local comparability vs doubling} 

Next we study the local comparability condition and its relationship to doubling.
Recall that any such comparison must be made in geometrically doubling spaces,
the only ones that support doubling measures.

\begin{example}\label{nondoub} 
For a very simple example of a nondoubling locally comparable
measure, on a geometrically doubling space, just take the space $X$ of 2 points $x$ and $y$ at distance 1, and let  $\mu := \delta_x$.
Then balls are either disjoint or the whole space (when $r >1$), so $C(\mu)=1$.
Actually, the same argument shows that $C(\nu)=1$ for every measure $\nu$ on $X$.
A variant of the preceding example, but admitting nondoubling measures with full
support, is given by  $\mathbb{N}$ with $d(m,n) = 1$ when $m\ne n$. If  $\mu$ is
any finite Borel measure on $\mathbb{N}$, then $C(\mu) = 1$. Alternatively, one can
just recall Example \ref{ultra}, noting that the preceding spaces are (discrete) ultrametric.
\end{example}

Example \ref{nondoub}  shows that a nontrivial locally comparable measure, can assign 
measure zero to some balls. For another difference between doubling and local comparability, note that if $\mu$ is doubling and 
$\mu\{x\} > 0$, then $x$ is an isolated point of $X$. 
This need not be the case when we only have local comparability, as can be seen by 
 choosing a measure with atoms
in an ultrametric space without isolated points.

Speaking loosely, the better the connectivity properties of the space, and the smaller the ``gaps"
or ``holes" in it, the more similar to doubling
are the measures satisfying local comparability. For instance, if $X$ is connected,
 then the local comparability of $\mu$ entails that no ball
has measure 0, and furthermore,
$\mu$ is continuous, that is, for every $x\in X$, $\mu\{x\} =0$. The next result
does not require $X$ to be geometrically doubling.

\begin{theorem}  \label{contmeas} If  a metric measure space $(X, d, \mu)$ is connected and
 $\mu$ satisfies a   
local comparability
 condition, then all balls have strictly positive measure, and for every $x\in X$, $\mu\{x\} =0$.
 \end{theorem} 
 
 \begin{proof}  If for some $r>0$ and some $x\in X$, $\mu B(x,r) = 0$, then by local comparability,
 for every $y\in B(x,r)$, $\mu B(y,r) = 0$, and by
  $\tau$-smoothness, $\mu Bl(x, r, r) = 0$. Define $Bl_1 := Bl(x, r, r)$ and for $n\ge 1$,
 $Bl_{n + 1} := Bl(Bl_n, r)$. Again by $\tau$-smoothness, for all $n\ge 1$, $\mu Bl_n = 0$,
 so $\mu \cup_n Bl_n = 0$. But $\cup_n Bl_n = X$ (this is well known and it follows
 from the fact that $\cup_n Bl_n$ is nonempty, open, and closed, so it is $X$; in topological
 terminology, connected metric spaces are chainable). Thus,  the nontriviallty of
 $\mu$ is contradicted.

Next, 
 suppose that for some $x\in X$ we have $\mu\{x\}  > 0$. Select $n\in\mathbb{N}\setminus
 \{0\}$
  and
 $r > 0$ such that $\mu(B(x,r)) < (1 + 1/n) \mu\{x\}$. Since $x$ is not isolated there exists a $y\in 
 B(x, r/3)\setminus \{x\}$. Now if $B(x, d(x,y))\cap B(y, d(x,y)) \ne \emptyset$, we can use these two balls to
 conclude that a local comparability condition cannot hold, since 
 $\mu B(x, d(x,y))/\mu B(y, d(x,y)) \ge \mu \{x\}/\mu B(y, d(x,y)) > n$ for $n$ arbitrary. 
 So assume  that $B(x, d(x,y))\cap B(y, d(x,y)) = \emptyset$. By connectivity, the closed
 ball $B^{cl}(x, d(x,y))$ is not open, whence there exists a $z\in B^{cl}(x, d(x,y))$ and a sequence
 $\{z_n\}^\infty_1$ in $\left(B^{cl}(x, d(x,y))\right)^c$ such that $\lim_n d(z_n, z) = 0$.
Select  $z_N \in B(z, d(x,y)/3)$. Then $d(x,z_N) > d(x, z) = d(x,y)$
 and $2 d(x,z_N) < r$, so 
  $z\in B(x, d(x,z_N))\cap B(z_N, d(x,z_N))$  and as before, 
  $\mu B(x, d(x,z_N))/\mu B(z_N, d(x,z_N)) > n$.
 \end{proof}

\begin{lemma} \label{blossombounds} Let $(X, d, \mu)$ be a metric measure space, 
$D$-geometrically doubling, and such that $\mu$ satisfies a 
local comparability
 condition. Then  $\mu(Bl(x,r, r))\le D \ C_\mu^3 \ \mu(B(x,r))$, 
 and $\mu(Blu(x,r, r))\le D^2 \  C_\mu^4 \ \mu(B(x,r))$.
\end{lemma}  

\begin{proof} Cover the ball $B(x, 2r)$  with at most $D$ balls of radius $r$, and disregard all such
balls having empty intersection with $Bl(x, r, r)$. Since $Bl(x, r, r)\subset B(x, 2r)$, this yields a cover 
$B(y_{i}, r), \dots , B(y_{M}, r)$ of $Bl(x, r, r)$ with $M\le D$, and 
by Lemma \ref{c2},
such that for every $1 \le k \le M$, 
$\mu B(y_{k}, r) \le C_\mu^3  \  \mu B(x,r)$. Hence 
$\mu Bl(x, r, r) \le D \  C_\mu^3 \ \mu (B(x,r))$. 

The second inequality is proven in the same way:  $Blu(x, r, r) \subset B(x, 3r)$;
cover  $B(x, 3r)$  with at most $D^2$ balls of radius $r$,  and disregard all such
balls  having empty intersection with $Blu(x, r, r)$. This yields a cover 
$B(y_{i}, r), \dots , B(y_{M}, r)$ of $Blu(x, r, r)$ with $M\le D^2$, and 
(applying Lemma \ref{c2} twice)
such that for every $1 \le k \le M$, 
$\mu B(y_{k}, r) \le C_\mu^4  \ \mu B(x,r)$. Hence 
$\mu Blu(x, r, r) \le D^2 \  C_\mu^4 \ \mu B(x,r)$. 
\end{proof}

It is obvious that in an arbitrary  metric measure space, boundedness of blossoms entails local comparability, since whenever
$d(x,y) < r$, $B(x,r) \subset Bl(y,r,r)$ and $B(y,r) \subset Bl(x,r,r)$. If additionally
the space is geometrically doubling, then the conditions are equivalent.

\begin{corollary} \label{blossombounds1} Let $(X, d, \mu)$ be a geometrically doubling metric measure space. The following are equivalent:

a) $\mu$ satisfies a 
local comparability
 condition. 
 
 b) There exists a constant $K_1 \ge 1$ such that for every $x\in X$ and every
 $r > 0$, $\mu(Bl(x,r, r))\le K_1 \ \mu(B(x,r))$.
 
 c) There exists a constant $K_2 \ge 1$ such that for every $x\in X$ and every
 $r > 0$, $\mu(Blu(x,r, r))$ $\le K_2 \ \mu(B(x,r))$.
 \end{corollary}

In a certain sense, it could be said that from the viewpoint of the maximal operator
the doubling condition is irrelevant, since in geometrically doubling spaces it can
be replaced by local comparability (by the preceding corollary together with the
Vitali covering lemma) and off geometrically doubling spaces, there are
no doubling measures.  However, if nearby points can always
be joined by  bounded chains of intersecting  balls with the same radius, then 
in a geometrically doubling space, local comparability implies doubling.
Hence, in many spaces of interest both conditions are equivalent.

\begin{lemma}\label{SS2} Let $(X, d, \mu)$ be a metric measure space, 
$D$-geometrically doubling, and such that $\mu$ satisfies a  
local comparability
 condition. Suppose there exists a $K > 1$ such that for all $x, z \in X$ and all $r >0$,
 whenever $d(x,z) < 2 r$ there exists a chain of balls $B(x, r) = B(y_0,r), \dots, B(y_m, r)$
 satisfying
$m\le K$, $z\in B(y_m, r)$ and for $j = 0, \dots, m-1$, $B(y_j, r) \cap  B(y_{j + 1}, r)\ne \emptyset$.
Then $\mu$ is  $D \  C_\mu^{2 K + 3}$- doubling. 
 \end{lemma}

\begin{proof} Cover $B(x, 2r)$ with at most $D$ balls of radius $r$. Of course, any ball
that does not intersect $B(x, 2r)$ can be disregarded, so suppose $B(w, r)$
is one of the balls in the cover, and let $z\in B(w,r)\cap B(x, 2r)$. Since $d(x,z) < 2 r$, there
is   an intersecting chain of balls $B(x, r) = B(y_1,r), \dots, B(y_m, r)$
such that
$m\le K$ and $z\in B(y_m, r)$. By repeated application of Lemma \ref{c2},
together with the fact that $d(y_m, z) < r$,
we conclude that $
\mu(B(z,r)) \le C_\mu^{2 K + 1} \ \mu(B(x,r)),
$
so $
\mu(B(w,r)) \le C_\mu^{2 K + 3} \ \mu(B(x,r)),
$
and the result follows.
\end{proof}

Next we indicate some conditions ensuring that the hypothesis of the previous lemma
holds.

\begin{definition} \label{amp}  A metric space has the {\it approximate midpoint property} if for every
$\varepsilon > 0$ and every pair of points $x,y$, there exists a point $z$ such
that $d(x,z), d(z,y) <  \varepsilon + d(x,y)/2$.
\end{definition}

\begin{definition} \label{quasi}A metric space is  {\it quasiconvex} if there exists a constant
$C\ge 1$ such that for every
 pair of points $x,y$, there exists a curve with $x$ and $y$ as endpoints, such
 that its length is bounded above by $C d(x,y)$. 
 \end{definition}

 We say that  $B(y_0,r), \dots, B(y_m, r)$ form an intersecting chain of balls if
 for $j = 0, \dots, m-1$, $B(y_j, r) \cap  B(y_{j + 1}, r)\ne \emptyset$.

\begin{corollary}\label{SS3} Let $(X, d, \mu)$ be a metric measure space, 
$D$-geometrically doubling, and such that $\mu$ satisfies a  
local comparability
 condition. If 
 either all balls are connected, or $X$ is quasiconvex, or it has the approximate midpoint property, then $\mu$ is doubling. 
 \end{corollary}

\begin{proof} Whenever we have a cover, we assume that all sets in it  intersect the set to be covered;
otherwise, 
 we disregard those not satisfying the condition.

 Fix $B_0 := B(x,r)$.  Suppose first that all the balls of $X$ are connected.
 Let the collection $\mathcal{C}:= \{B(x_1, r), \dots, B(x_M, r)\}$
 be a cover of $B(x, 2r)$ with
$M\le D$, and write $B_i := B(x_i, r)$. Let $\mathcal{C}^\prime$ be the collection of all balls $B_i$  in $\mathcal{C}$
for which there is an intersecting chain of balls starting at $B_0$ and finishing with $B_i$. 
Then $\mathcal{C} = \mathcal{C}^\prime$, for otherwise the union of all balls in $\mathcal{C}^\prime$ and the
union  of all balls in $\mathcal{C}\setminus \mathcal{C}^\prime$ would form a disconnection of $B(x, 2r)$.
Adding $B_0$ to the balls in the cover, we see that the maximal lenght of any chain is $M + 1 \le D + 1$.

Suppose next that $X$ is quasiconvex with constant $C_X$ (any
constant satisfying Definiton \ref{quasi}; there might
 not be a smallest one).  Choose $y $ with $d(x,y) < 2r$. Then
there exists a curve $c: [0,1]\to X$ starting at $x$ and finishing at $y$ (that is,
$c(0) = x$, $c(1) = y$) such that its length $L(c)$ satisfies $L(c) < C_X 2r$.
Let $K = [L(c)/r] + 1$, where $[L(c)/r]$ denotes the integer part of $L(c)/r$. 
Divide $c$ into $K$ subsegments of equal length, with endpoints $x_0 = x, x_1$, ...,
$x_{K-1}, x_K = y$. Then the balls   $\{B(x_0, r), \dots, B(x_{K - 1}, r)\}$
form an intersecting chain
of length $K \le [2 C_X] + 1$.

The argument for the case where $X$ has the approximate midpoint property
is similar and simpler, so we omit it.
\end{proof}

 If none of the  conditions in the preceding result hold, the equivalence between local comparability and doubling can fail,
 even for arc-connected spaces.

\begin{theorem} \label{arcconneted} There exists an arc-connected, geometrically doubling 
metric measure space $(X, d, \mu)$, such that $X\subset \mathbb{R}^2$,  $d$ is defined by the restriction
of the $\ell_\infty$-norm to $X$, and 
$\mu$ satisfies a 
local comparability
 condition but is not doubling. 
\end{theorem}

\begin{proof} On $\mathbb{R}$ set $d\nu (x) = dx$ for $x\le 1$, and  $d\nu (x) = x dx$ for $x\ge 1$ 
(we mention that  $\nu$ does not satisfy a local comparability condition, since
$\lim_{x\to\infty} \nu([0, x])/  \nu([-x, 0] )=\infty$, so in particular it is not doubling). 
Next we define an embedding $f : \mathbb{R} \to \mathbb{R}^2$
as follows: $f: (-\infty, -1] \to [0,\infty) \times \{0\}$ is given by $f(t) = (- 1 -t , 0)$, 
$f: [-1, 0] \to  \{0\}\times [0,1]$ is given by $f(t) = (0,1 + t)$, and 
$f: [0,\infty) \to [0,\infty) \times \{1\}$, by $f(t) = (t  , 1)$.
Then we set  $X:= f(\mathbb{R}) \subset \mathbb{R}^2$, with $d$ on $X$ 
defined by the $\ell_\infty$ norm on the plane:
$d((a,b), (c,d)) =  \max\{|a-c|,|b-d|\}$. Let $\mu$ be the pushforward measure $f_*\nu$, so
$\mu A := \nu (f^{-1} (A))$.   Then $\lim_{x\to\infty} \mu(B((x,0), 2)/ \mu(B((x,0), 1) =\infty$, so
$\mu$ is not doubling. However,
 since  $B((x,0), t) = B((x,1), t)$ for $t > 1$, while $B((x,0), t) \cap B((x,1), t) = \emptyset$ for $0 < t \le 1$
and $x\ge 1$, 
it  is not difficult to see that $\mu$ satisfies a local
comparability condition. More precisely, let $S:= f (-\infty, 1]$ and note that on $S$, $\mu$ is just length.
Thus, for balls centered at  points $x = (t, 0)$ or $x = (0, t)$ with $r\le 1$, the result is clear, while if $r > 1$, then $x = (t, 1)$
is also a center of $B((t,0), r))$, so it is enough to consider the case
$x = (t, 1)$, $t  \ge 0$. Since 
$$
\mu(S \cap B((t,1), r)) 
\ge
\mu(S \cap B((t + r,1), r)), 
$$
we have 
$$
\frac{\mu(B((t + r,1), r))}{ \mu(B((t,1), r))} 
= 
 \frac{\mu(S \cap B((t + r,1), r)) + \mu(S^c \cap B((t + r,1), r)) }{\mu(S \cap B((t,1), r)) + \mu(S^c \cap B((t,1), r))}
$$
$$
\le 
 \frac{\mu(S^c \cap B((t + r,1), r)) }{\mu(S^c \cap B((t,1), r))}
\le 
 \frac{\int_t^{t + 2 r} u du }{\int_t^{t + r} u du }
\le 4.
$$
(A more involved argument yields the optimal constant $C(\mu) = 2$).
\end{proof}

Note that in the space $X= f(\mathbb{R})$ defined above, if $x\ge 2$, then $Blu((x,0), 1, 1)$ does not
contain any ball strictly larger than $B((x,0), 1)$; in particular, for all 
$t > 0$, $B((x,0), 1 + t)\not\subset Blu((x,0), 1, 1)$. So even in arc-connected spaces,
blossoms and balls can be rather
different.  On the other hand, if there exists a fixed $t > 0$ such that for every $x\in X$ and every
$r > 0$, $B(x, (1 + t) r) \subset Bl(x, r, r)$, then boundedness of blossoms entails doubling.

\section{Remarks on spaces that may fail to be geometrically doubling}

Boundedness of blossoms does not imply that $X$ is geometrically doubling.

\begin{theorem}  \label{bb}There exists  a metric measure space $(X, d, \mu)$, such that $(X, d)$ is
not geometrically doubling,  $\mu$ satisfies a   
local comparability
 condition, and blossoms boundedly. 
 Furthermore, $\mu$ can be chosen to have full support.
 \end{theorem} 

\begin{proof}  We use the infinite broom $B\subset \mathbb{R}^2$, 
 $B:=\cup_{n\in \mathbb{N}} \{(x, nx): x\ge 0\}$, with $d$  the path metric. 
Now for $n\ge 1$, 
 let $z_n\in \{(x, nx): x\ge 0\}$ be the only point that satisfies
$d(0, z_n)=1/n$. Set $X :=\{0\}  \cup \{z_n : n \ge 1\} \subset B$, with the distance inherited from $B$.
To see that $(X, d)$ is not geometrically doubling, note that $B(0,1/n)$ contains the following disjoint
balls: $B(0,1/(2n))$, $B(z_{n + 1}, 1/(2n))$, $\dots$, $B(z_{2n}, 1/(2n))$.
Next, let $\mu = \delta_0$ be the point mass at the origin. It two balls $B_1, B_2$ centered at points of $X$
intersect, they both must contain $0$, so $\mu B_1 = \mu B_2 = 1$, and hence $C(\mu) = 1$. Blossoming
at $0$ does not change the mass, while blossoming at other points leads to either not increasing
the original ball, or  not increasing its measure: If $0\notin B(z_{k}, r)$, then
$Blu(z_{k}, r, r) = B(z_{k}, r)$, and if $0\in B(z_{k}, r)$, then $\mu Blu(z_{k}, r, r) = \mu B(z_{k}, r) = 1$.

One can easily
modify $\mu$ so that it also has full support: In addition to the Dirac delta at the origin,  give mass $2^{-n}$
to each $z_n$, and argue essentially as before. 
\end{proof} 

Beyond geometrically doubling metric spaces, it is unclear to me 
whether or not local comparability suffices to obtain boundedness
results for the Hardy-Littlewood maximal operator. Of course, if blossoms are
bounded, then the usual Vitali covering argument works.  Nevertheless, in specific and natural
examples, such as volume in  hyperbolic spaces (cf. \cite{Str}, \cite{LiLo}), blossoms
are not bounded, and still  weak type $(1,1)$ bounds hold for the centered maximal
operator (the uncentered operator is in general unbounded, as can be seen by
considering just one Dirac delta in the hyperbolic plane). A rather different 
instance of this phenomenon is presented in \cite[Theorem 1.5]{NaTa}, where boundedness of the
maximal operator is obtained for the infinite, rooted $k$-ary tree with the
standard graph metric.
In these examples, however, 
one not only has local, but global comparability with constant 1, since the
measure of balls only depends on the radius, and not the center.

While writing this paper I found \cite{SoTr}, where local comparability is considered,
under the name of  ``equidistant comparability property", in the specific case of connected
graphs with all vertices having finite degree. 
Here $X$ is the set of vertices $V$ of a graph,  $d(x,y)$ is defined as the smallest
number of edges one needs to traverse in order  to go from $x$ to $y$, and $\mu$ is the counting
measure. Since all vertices have finite degree,   the measure of all balls is finite, and we
are within the general framework considered in the present paper. For instance, it 
is easy to check that for the metric measure spaces considered here, 
the centered maximal operator acting on just one Dirac delta is
weak (1,1) bounded, with bound $C_\mu$ (this is so even if the Dirac delta
is placed at a point outside the support of $\mu$). The special case of this
result for
graphs appears in \cite[Proposition 2.10]{SoTr}.


\begin{thebibliography}{WWW}


\bibitem[A]{A} J.M. Aldaz,   \textit{The  weak type $(1,1)$ bounds for the 
maximal function associated to cubes
grow to infinity with the dimension}, Ann. of Math. (2) \textbf{173}
(2011), no. 2, 1013--1023.


\bibitem[A2]{A2} J.M. Aldaz,   \textit{The Stein Str\"omberg Covering Theorem in metric spaces}, arXiv:1605.05596. 

\bibitem[AlPe]{AlPe} J.M. Aldaz and J. P\'erez L\'azaro, \textit{Behavior of weak type bounds for high
dimensional maximal operators defined by certain radial measures},
Positivity \textbf{15} (2011), 199--213.




\bibitem[AuCa]{AuCa} P. Auscher and M.J. Carro, \textit{Transference for radial multipliers and dimension free estimates}, Trans. Amer. Math. Soc.
\textbf{342} (1994), no. 5, 575--593.



\bibitem[Bou1]{Bou1} J. Bourgain, \textit{On high-dimensional maximal functions associated to
convex bodies}, Amer. J. Math. \textbf{108} (1986), no. 6,
1467--1476.

\bibitem[Bou2]{Bou2} J. Bourgain, \textit{On the $L\sp p$-bounds for maximal functions
associated to convex bodies in $R\sp n$}. Israel J. Math.
\textbf{54} (1986), no. 3, 257--265.


\bibitem[Bou3]{Bou3} J. Bourgain, \textit{On the Hardy-Littlewood maximal function for the cube},
 Israel J. Math. 203 (2014), no. 1, 275--293.

\bibitem[Ca]{Ca} A. Carbery, \textit{An almost-orthogonality principle with
applications to maximal functions associated to convex bodies},
Bull. Amer. Math. Soc. (N.S.) \textbf{14} (1986), no. 2, 269--273.




\bibitem[CriSjo]{CriSjo}
A. Criado and P. Sj\"ogren, \textit{ Bounds for maximal functions associated with rotational invariant measures in high dimensions},  J. Geom. Anal.  \textbf{24} (2014), 595--612.


\bibitem[CoWe1]{CoWe1}  R. R. Coifman, .  G. Weiss, {\em Analyse harmonique non-commutative sur certains espaces homog\`enes.  \'Etude de certaines int\'egrales singuli\`eres.} Lecture Notes in Mathematics, Vol. 242. Springer-Verlag, Berlin-New York, 1971.

\bibitem[CoWe2]{CoWe2} R. R. Coifman, .  G. Weiss, {\em Extensions of Hardy spaces and their use in analysis.} Bull. Amer. Math. Soc. 83 (1977), no. 4, 569--645. 



\bibitem[Fre]{Fre} D.H.  Fremlin, {\em Real valued measurable cardinals.} 
Version of 19.9.09.



\bibitem[Gra]{Gra} L. Grafakos,  {\em Classical Fourier analysis. Third edition.}
Graduate Texts in Mathematics, 249. Springer, New York, (2014).


\bibitem[GraB]{GraB} L. Grafakos, {\em Modern Fourier analysis. Second edition.}
Graduate Texts in Mathematics, 250. Springer, New York, 2009.




\bibitem[He]{He} J. Heinonen,  {\em Lectures on analysis on metric spaces.} Universitext. Springer-Verlag, New York, 2001. 


\bibitem[HKST]{HKST}  J. Heinonen, P.  Koskela, N.  Shanmugalingam, J. T Tyson, 
{\em Sobolev spaces on metric measure spaces. 
An approach based on upper gradients.} New Mathematical Monographs, 27. Cambridge University Press, Cambridge, 2015.  

\bibitem[IaSt]{IaSt} A.S. Iakovlev and J.-O. Str\"omberg,
 \textit{Lower bounds for the weak type $(1,1)$
estimate for the maximal function associated
to cubes in high dimensions}, Math. Res. Letters 20 (2013) no. 5, 907--918.

\bibitem[Li]{Li} E. Lindenstrauss,  {\em Pointwise theorems for amenable groups.} Invent. Math. 146 (2001), no. 2, 259–295. 

 \bibitem[LiLo]{LiLo} H.-Q. Li and N. Lohou\'e,
\textit{Fonction maximale centr\'ees de Hardy-Littlewood sur les
espaces hyperboliques},  Ark. f\"or Mat. 50 (2012), no. 2,  359--378.

 
 \bibitem[LuSa]{LuSa}  J. Luukkainen, E. Saksman,   {\em Every complete doubling metric space carries a doubling measure.}
 Proc. Amer. Math. Soc. 126 (1998), no. 2, 531--534.

\bibitem[Mu]{Mu} D. M\"{u}ller, \textit{A geometric bound for maximal functions
associated to convex bodies}, Pacific J. Math. \textbf{142} (1990),
no. 2, 297--312.
 


\bibitem[NaTa]{NaTa} A. Naor and T. Tao, \textit{Random martingales and localization of maximal inequalities},
 J. Funct. Anal. \textbf{259} (2010), no. 3, 731--779.
 
  \bibitem[SoTr]{SoTr}  J. Soria and P. Tradacete,  {\em Geometric properties of infinite graphs and the Hardy-Littlewood maximal operator.} arXiv:1602.01029


\bibitem[St1]{St1} E.M. Stein, \textit{The development of square functions in the work of A.
Zygmund. Bull. Amer. Math. Soc}, (N.S.) \textbf{7} (1982), no. 2,
359--376.

\bibitem[St2]{St2} E.M. Stein, \textit{Three variations on the theme of maximal
 functions}, Recent progress in Fourier analysis (El Escorial, 1983), 229--244,
 North-Holland Math. Stud., 111, North-Holland, Amsterdam, 1985.

\bibitem[St3]{St3} E.M. Stein, \textit{Harmonic analysis: real-variable methods,
 orthogonality, and oscillatory integrals}, Princeton University Press, Princeton, NJ, 1993.


 \bibitem[StSt]{StSt} E. M.  Stein, J. O. Str\"{o}mberg,  {\em Behavior of maximal functions in $R\sp{n}$
 for large $n$.} Ark. Mat. 21 (1983), no. 2, 259--269.

  \bibitem[Str]{Str} J. O. Str\"{o}mberg,  {\em Weak type $L^1$ estimates for maximal functions on noncompact symmetric spaces.} 
Ann. of Math. (2) 114 (1981), no. 1, 115--126. 
 
 \bibitem[To]{To} X. Tolsa, {\em Analytic capacity, the Cauchy transform, and non-homogeneous 
 Calder\'on-Zygmund theory.} Progress in Mathematics, 307. Birkhauser/Springer,  2014.
 
 \bibitem[Web]{Web} Webster,  R. J. {\em Convexity} (Oxford University Press, 1997).

 


\end{thebibliography}
\end{document}